\documentclass[11pt,article,reqno]{amsart}
\usepackage{amssymb,url}
\usepackage{graphicx}
\usepackage{amsthm}
\usepackage{amsmath}

 \newtheorem{thm}{Theorem}[section]
 \newtheorem{cor}[thm]{Corollary}
 \newtheorem{lem}[thm]{Lemma}
 \newtheorem{prop}[thm]{Proposition}
 \theoremstyle{definition}
 
 \theoremstyle{remark}
 \newtheorem{rem}[thm]{Remark}
 
 \newtheorem{ques}[thm]{Question}
 \numberwithin{equation}{section}

\newcommand{\C}{\mathbb{C}}

\newcommand{\N}{\mathbb{N}}
\newcommand{\Z}{\mathbb{Z}}

\newcommand{\T}{\mathbb{T}}
\newcommand{\re}{\mathrm{Re}}
\newcommand{\Br}{\mathrm{Br}}
\newcommand{\ran}{\mathrm{ran}}

\newcommand{\Chi}{\mathrm{Chi}}
\newcommand{\Par}{\mathrm{par}}
\newcommand{\Des}{\mathrm{Des}}
\newcommand{\Gen}{\mathrm{Gen}}
\newcommand{\Lea}{\mathrm{Lea}}
\newcommand{\roo}{\mathrm{root}}

\newcommand{\irB}{\mathcal{B}}

\newcommand{\irE}{\mathcal{E}}
\newcommand{\irG}{\mathcal{G}}
\newcommand{\irH}{\mathcal{H}}
\newcommand{\irJ}{\mathcal{J}}
\newcommand{\irK}{\mathcal{K}}
\newcommand{\irL}{\mathcal{L}}
\newcommand{\irM}{\mathcal{M}}
\newcommand{\irP}{\mathcal{P}}

\newcommand{\irT}{\mathcal{T}}
\newcommand{\Sl}{S_{\boldsymbol{\lambda}}}

\newcommand{\Ww}{W_{\boldsymbol{w}}}

\begin{document}

\title[As. beh. and cyclic prop. of weighted shifts on dir. trees]{Asymptotic behaviour and cyclic properties of weighted shifts on directed trees}

\author{Gy\"orgy P\'al Geh\'er}

\address{Bolyai Institute, University of Szeged, Aradi v\'ertan\'uk tere 1., H-6720, Szeged, Hungary}
\address{MTA-DE "Lend\"ulet" Functional Analysis Research Group, Institute of Mathematics, University of Debrecen, H-4010 Debrecen, P.O. Box 12, Hungary}

\email{gehergy@math.u-szeged.hu or gehergyuri@gmail.com}

\thanks{
This research was realized in the frames of T\'AMOP 4.2.4. A/2-11-1-2012-0001 ''National Excellence Program - Elaborating and operating an inland student and researcher personal support system''. The project was subsidized by the European Union and co-financed by the European Social Fund.\\
The author was also supported by the "Lend\"ulet" Program (LP2012-46/2012) of the Hungarian Academy of Sciences.}
\subjclass{Primary 47A16, 47B37}

\keywords{Bounded weighted shifts on directed trees, cyclic operator, contraction, asymptotic limit.}

\date{}

\begin{abstract}
In this paper we investigate a new class of bounded operators called weighted shifts on directed trees introduced recently in \cite{tree-shift}. This class is a natural generalization of the so called weighted bilateral, unilateral and backward shift operators. In the first part of the paper we calculate the asymptotic limit and the isometric asymptote of a contractive weighted shift on a directed tree and that of the adjoint. Then we use the asymptotic behaviour and similarity properties in order to obtain cyclicity results. We also show that a weighted backward shift operator is cyclic if and only if there is at most one zero weight.
\end{abstract}

\maketitle

\section{Introduction}

The classes of the so-called weighted bilateral, unilateral or backward shift operators (\cite{Ni, Shields}) are very useful for an operator theorist. 
Besides normal operators these are the next natural classes on which conjectures could be tested. 
Recently Z. J. Jab\l onski, I. B. Jung and J. Stochel defined a natural generalization of these classes in \cite{tree-shift}, called weighted shifts on directed trees. 
Among others, they were interested in hyponormality, co-hyponormality, subnormality etc., and they provided many examples for several unanswered questions.
They continued their research in several papers, see \cite{tree-shift-,tree-shift--,tree-shift_,tree-shift__,Tre}.

In this paper we will study cyclic properties of bounded (mainly contractive) weighted shift operators on directed trees. 
First, we will explore their asymptotic behaviour, and as an application we will obtain some results concerning cyclicity.
In the next few pages we give some auxiliary definitions which will be essential throughout this investigation.

\subsection{Directed trees}
Concerning the definition of a directed tree we refer to the monograph \cite{tree-shift}. 
Throughout this paper $\irT=(V,E)$ will always denote a \textit{directed tree}, where $V$ is a non-empty (usually infinite) set and $E\subseteq V\times V\setminus \{(v,v)\colon v \in V\}$. 
We call an element of $V$ and $E$ a \textit{vertex} and a \textit{(directed) edge} of $\irT$, respectively. 
If we have an edge $(u,v)\in E$, then $v$ is called a \textit{child} of $u$, and $u$ is called the \emph{parent} of $v$.
The set of all children of $u$ is denoted by $\Chi_\irT(u)=\Chi(u)$, and the symbol $\Par_\irT(v) = \Par(v)$ stands for $u$. 
We will also use the notation $\Par^k(v) = \underbrace{\Par(\dots (\Par}_{k \text{-times}}(v))\dots)$ when it makes sense, and $\Par^0$ will be the identity map. 

If a vertex has no parent, then we call it a \textit{root} of $\irT$. 
A directed tree is either rootless or has a unique root (see \cite[Proposition 2.1.1]{tree-shift}) which, in this case, will be denoted by $\roo_\irT = \roo$. 
We will use the notation 
\[ 
V^\circ = \left\{ 
\begin{matrix}
V\setminus \{\roo\} & \text{if } V \text{ has a root,}\\
V & \text{elsewhere.}
\end{matrix}\right.
\]
If a vertex has no children, then we call it a \textit{leaf}, and $\irT$ is \textit{leafless} if it has no leaves. 
The set of all leaves of $\irT$ will be denoted by $\Lea(\irT)$. 
Given a subset $W\subseteq V$ of vertices, we put $\Chi(W) = \cup_{v\in W}\Chi(v)$, $\Chi^0(W) = W$ and $\Chi^{n+1}(W) = \Chi(\Chi^{n}(W))$ for all $n\in\N$. 
The set $\Des_\irT(W) = \Des(W)=\bigcup_{n=0}^\infty \Chi^{n}(W)$ is called the \textit{descendants} of the subset $W$, and if $W = \{u\}$, then we simply write $\Des(u)$.
If $n\in\N_0 (:=\N\cup\{0\})$, then the set $\Gen_{n,\irT}(u)=\Gen_n(u)=\bigcup_{j=0}^n\Chi^j(\Par^j(u))$ is called the \textit{$n$th generation} of $u$ and $\Gen_\irT(u)=\Gen(u)=\bigcup_{n=0}^\infty\Gen_n(u)$ is the \textit{(whole) generation} or the \textit{level} of $u$. 

From the equation 
\begin{equation}\label{level_eq}
V = \bigcup_{n=0}^\infty\Des(\Par^n(u))
\end{equation}
(see \cite[Proposition 2.1.6]{tree-shift}), one can easily see that the different levels can be indexed by the integer numbers (or by a subset of the integers) in such a way that if a vertex $v$ is in the $k$th level, then the children of $v$ are in the $(k+1)$th level, and whenever $\Par(v)$ is defined, it lies in the $(k-1)$th level.

\subsection{Bounded weighted shifts on directed trees}
The complex Hilbert space $\ell^2(V)$ is the usual space of all square summable complex functions on $V$ with the standard innerproduct
\[ \langle f,g\rangle = \sum_{u\in V} f(u)\overline{g(u)} \qquad (f,g\in\ell^2(V)). \]
For $u\in V$ we define $e_u(v)=\delta_{u,v}\in\ell^2(V)$, where $\delta_{u,v}$ is the Kronecker delta. 
Obviously the set $\{e_u\colon u\in V\}$ is an orthonormal base. We will refer to $\ell^2(W)$ as the subspace (i.e. closed linear manifold) $\vee\{e_w\colon w\in W\}$ for any subset $W\subseteq V$, where the symbol $\vee\{\dots\}$ stands for the generated subspace.

Let $\boldsymbol{\lambda} = \{\lambda_v\colon v\in V^\circ\}\subseteq\C$ be a set of \textit{weights} satisfying the following condition: $\sup\left\{\sqrt{\sum_{v\in\Chi(u)}|\lambda_v|^2}\colon u\in V\right\}<\infty$.
Then the \textit{weighted shift on the directed tree $\irT$} is the operator defined by
\[\Sl\colon \ell^2(V)\to \ell^2(V), \quad e_u\mapsto \sum_{v\in\Chi(u)} \lambda_v e_v. \]
By \cite[Proposition 3.1.8]{tree-shift} this defines a bounded linear operator with norm 
\[
\|\Sl\| = \sup\left\{\sqrt{\sum_{v\in\Chi(u)}|\lambda_v|^2}\colon u\in V\right\}.
\]

We will consider only bounded weighted shifts on directed trees, especially contractions (i.e. $\|\Sl\|\leq 1$) in certain parts of the paper. 
We recall that every $\Sl$ is unitarily equivalent to $S_{|\boldsymbol{\lambda}|}$ where $|\boldsymbol{\lambda}|:=\{|\lambda_v|\colon v\in V^\circ\}\subseteq[0,\infty)$ (see \cite[Theorem 3.2.1]{tree-shift}). 
Moreover, the unitary operator $U$ with $S_{|\boldsymbol{\lambda}|} = U\Sl U^*$ can be chosen such that $e_u$ is an eigen-vector of $U$ for every $u\in V$. 
It is also proposed in \cite[Proposition 3.1.6]{tree-shift} that if a weight $\lambda_v$ is zero, then the weighted shift on this directed tree is a direct sum of two other weighted shifts on directed trees. 
In view of these facts, this article will exclusively consider weighted shifts on directed trees with positive weights (i.e. $\lambda_v > 0$ for every $v\in V^\circ$), if we do not say otherwise.

The positivity of weights imply that every vertex has countably many children. 
Thus, by (\ref{level_eq}), $\ell^2(V)$ is separable.

\subsection{Asymptotic behaviour}
Let $\irH$ be a complex Hilbert space and let us denote the algebra of bounded linear operators on it by $\irB(\irH)$. If $T\in\irB(\irH)$ is a contraction, then the sequences $\{T^{*n}T^n\}_{n=1}^\infty$ and $\{T^{n}T^{*n}\}_{n=1}^\infty$ of positive contractions are decreasing. Therefore they have unique limits in the strong operator topology (SOT):
\[ A = A_T = \lim_{n\to\infty} T^{*n}T^n \quad\text{ and }\quad A_* = A_{T^*} = \lim_{n\to\infty} T^{n}T^{*n}. \]
The operator $A$ is the \textit{asymptotic limit} of $T$ and $A_*$ is the \textit{asymptotic limit} of the adjoint $T^*$. 

The vector $h\in\irH$ is called \textit{stable} for the contraction $T\in\irB(\irH)$ if the orbit of $h$ converges to 0, i.e. $\lim_{n\to\infty} \|T^n h\| = 0$ or equivalently $h \in \ker(A_T)$. The set $\ker(A_T)$ of all stable vectors is usually denoted by $\irH_0(T)$ and called the \emph{stable subspace} of $T$. We recall that the stable subspace is hyperinvariant for $T$ (i.e. invariant for every $C\in\irB(\irH)$ which commutes with $T$), which can be verified easily.

Contractions can be classified according to the asymptotic behaviour of their iterates and the iterates of their adjoints. Namely, $T$ is \textit{stable} or \textit{of class $C_{0\cdot}$} when $\irH_0(T) = \irH$, in notation: $T\in C_{0\cdot}(\irH)$. If the stable subspace consists only of the null vector, then $T$ is \textit{of class $C_{1\cdot}$} or $T\in C_{1\cdot}(\irH)$. In the case when $T^*\in C_{i\cdot}(\irH)$ ($i=0$ or 1), we say that $T$ is \textit{of class $C_{\cdot i}$}. Finally, the class $C_{ij}(\irH)$ stands for the intersection $C_{i\cdot}(\irH)\cap C_{\cdot j}(\irH)$.

By $\{\dots\}^-$ we mean the closure of a set. 
We recall that the operator $X\in\irB(\irH,\ran(A_T)^-) = \irB(\irH,\irH_0(T)^\perp)$, $Xh = A_T^{1/2}h$ acts as an intertwining mapping in a canonical realization of the so called isometric asymptote of the contraction $T$. 
This and the unitary asymptote are very efficient tools in the theory of Hilbert space contractions. Here we only give the specific realization but we note that there is a more general setting (\cite{BerKer}). 
There exists a unique isometry $U = U_T \in \irB(\ran(A_T)^-)$ such that $XT = UX$ holds. 
The isometry $U$ (or sometimes the pair $(X,U)$) is the \textit{isometric asymptote} of $T$. 
For a detailed study of isometric and unitary asymptotes, including other useful realizations (e.g. with the *-residual part of the minimal unitary dilation of $T$), we refer to \cite[Chapter IX]{NFBK} and \cite{KerchyTotik}. 
(We notice that in some papers about unitary asymptotes, $X$ is denoted by $X_+$ and the intertwining mapping of the unitary asymptote is denoted by $X$).

There are several applications for the isometric (and unitary) asymptotes.
They play an important role in the hyperinvariant subspace problem, similarity problems and operator models (see e.g. \cite{BerKer,Ca,Du,Ke_gen_Toep,Ke_isom_as,DuKu,NFBK}).

In the next section we show how the isometric asymptote can be used in order to obtain cyclicity results. Section 3 and 4 are technical parts of the paper devoted to calculating the asymptotic limits $A$ and $A_*$ and the isometric asymptotes $U$ and $U_*$ of the contractive $\Sl$ and $\Sl^*$, respectively. After that in Section 5 we characterize cyclicity of weighted backward shift operators. Finally, in the last three sections we investigate cyclic properties of weighted shifts on directed trees and their adjoints, using some similarity results and the results of Section 3-4.


\section{Cyclic properties of contractions and their isometric asymptote}

This section is devoted to explaining how the asymptotic behaviour can be used in order to obtain cyclicity results for contractions. 
We call the operator $T\in\irB(\irH)$ \textit{cyclic} if there exists a vector such that 
\[ \irH_{T,h} := \vee\{T^n h\colon n\in\N_0\} = \{p(T) h \colon p\in \irP_\C\}^-  = \irH, \]
where $\irP_\C$ denotes the set of all complex polynomials. Such a vector $h\in\irH$ is called a \textit{cyclic vector} for $T$. 

The vector $h\in\irH$ is \textit{hypercyclic} for $T$ if we have
\[ \{T^n h\colon n\in\N_0\}^- = \irH. \]
Then the operator $T$ is \textit{hypercyclic}. By a \emph{nilpotent} operator $N\in\irB(\irH)$ we mean that there exists a $k\in\N$ such that $N^k = 0$.

If $T$ is cyclic and has dense range, then $h$ is cyclic if and only if $Th$ is cyclic. This and a consequence are stated in the next lemma for Hilbert spaces, but we note that in Banach spaces the proof would be the same. This also shows that the set of cyclic vectors span the whole space, when $\ran T$ is dense. In fact, this is always true, see \cite{GeherL} for an elementary proof.

\begin{lem}\label{cyclic_denserenage_nilp_lem}
\begin{itemize}
\item[\textup{(i)}] if $T, Q, Y\in \irB(\irH)$, $Y$ has dense range, $YT = QY$ holds and $f$ is cyclic (or hypercyclic, resp.) for $T$, then $Yf$ is cyclic (or hypercyclic, resp.) for $Q$.
\item[\textup{(ii)}] If a dense range operator $T\in\irB(\irH)$ has a cyclic vector $f$, then $T f$ is also a cyclic vector.
\item[\textup{(iii)}] If the cyclic operator $T\in\irB(\irH)$ has dense range, and $N\in\irB(\C^n)$ ($n\in\N$) is cyclic and nilpotent, then $T\oplus N$ is also cyclic.
\end{itemize}
\end{lem}

\begin{proof}
(i) Of course $Y p(T) = p(Q) Y$ holds for all $p\in\irP_\C$. 
Let us assume that $f$ is cyclic for $T$, i.e. $\{p(T)f\colon p\in\irP_\C\}$ is dense in $\irH$. 
Then $\{Y p(T)f \colon p\in\irP_\C\} = \{p(Q) Y f \colon p\in\irP_\C\}$ is also dense, which implies that $Y f$ is cyclic for $Q$. 
The hypercyclic case is very similar.

(ii) This follows from (i) by choosing $Q=Y=T$.

(iii) Let us take a cyclic vector $f \in \irH$ for $T$ and a cyclic vector $e \in \C^n$ for $N$.
We show that $f\oplus e$ is cyclic for the orthogonal sum $T\oplus N$. 
Of course $\vee\{T^k f \oplus N^k e \colon k\geq n\} = \vee\{T^k f \oplus 0 \colon k\geq n\} = \irH$. 
Therefore $0 \oplus N^j e \in \vee\{T^k f \oplus N^k e \colon k \in \N_0\}$ for every $0 \leq j < n$, which implies that $f\oplus e$ is a cyclic vector.
\end{proof}

The previous and the next lemma will be used several times throughout this paper. 

\begin{lem} \label{is_as_forcyclem}
\begin{itemize}
\item[\textup{(i)}] If $T\in C_{1\cdot}(\irH)$ is a contraction and the isometric asymptote $U$ has no cyclic vectors, then neither has $T$,
\item[\textup{(ii)}] if $T\in C_{1\cdot}(\irH)$ is a contraction and the adjoint of the isometric asymptote $U^*$ has a cyclic vector $g$, then $A^{1/2} g$ is cyclic for $T^*$,
\item[\textup{(iii)}] if $T\in C_{\cdot 1}(\irH)$ is a contraction and the adjoint of the isometric asymptote $U_*^*$ has a cyclic vector $g$, then $A_*^{1/2} g$ is cyclic for $T$.
\end{itemize}
\end{lem}

We omit the proofs, since every point is a straightforward consequence of (i) in Lemma \ref{cyclic_denserenage_nilp_lem}. 
We close this section with discussing the cyclicity of contractive $C_{\cdot 1}$-class weighted bilateral shift operators. 
Let $\boldsymbol{w} = \{w_k\}_{k=-\infty}^\infty$ be a sequence such that $0<|w_k|\leq 1$. 
The weighted bilateral shift operator (of multiplicity one) $S_{\boldsymbol{w}}\in\irB(\ell^2(\Z))$ is defined by $S_{\boldsymbol{w}} e_k = w_{k+1} e_{k+1}$ $(k\in\Z)$ (trivially $S_{\boldsymbol{w}}$ is a weighted shift on the directed tree $(\Z,E)$ where $E = \{(k,k+1)\colon k\in\Z\}$).
Thus $S_{\boldsymbol{w}}^* e_k = \overline{w_k} e_{k-1}$ $(k\in\Z)$.
An easy calculation shows that the following equation holds:
\[ 
A_* e_k = \Big(\prod_{j\leq k} |w_j|^2\Big) e_k \quad (k\in\Z). 
\]
This means that $S_{\boldsymbol{w}}\in C_{\cdot 0}(\ell^2(\Z))\cup C_{\cdot 1}(\ell^2(\Z))$. 
In case when we have $S_{\boldsymbol{w}} \in C_{\cdot 1}(\ell^2(\Z))$, the isometric asymptote of $S_{\boldsymbol{w}}^*$: $U_*e_k = e_{k-1}$ $(k\in\Z)$, is unitarily equivalent to the simple bilateral shift operator. 
Since $U_*^*$ is clearly cyclic, every contractive $C_{\cdot 1}$-class weighted bilateral shift operator is cyclic. 

In this paper an orthogonal sum of weighted bilateral shift operators $\sum_{j\in\irJ} \oplus S_{\boldsymbol{w}^{(j)}}$ will be also called a weighted bilateral shift operator. 
The multiplicity of such an operator is the cardinality of $\irJ$ which is denoted by $\#\irJ$. 
We define the multiplicity of unilateral or backward shift operators very similarly.

If $S_{\boldsymbol{w}} \in C_{\cdot 0}(\ell^2(\Z))$, then we cannot use the above technique.
We note that bilateral shift operators which do not have any cyclic vector exist.
The first example was given by B. Beauzamy in \cite{non-cyclic_bil_shift} (see also \cite[Proposition 42]{Shields}).
As far as we know there is no characterization for cyclic weighted bilateral shift operators which is quite surprising, since for other cyclic type properties we can find characterizations (for example hyper- or supercyclicity can be found in \cite{hypercyclic_bil} and \cite{supercyclic_bil}). 
In our opinion it is a challenging problem to give this characterization for cyclicity.


\section{Asymptotic limits of contractive weighted shifts on directed trees}

In this section we prove several results concerning the asymptotic behaviour of contractive weighted shifts on directed trees.
They will be crucial later.
The powers of $\Sl$ were calculated in \cite[Lemma 2.3.1]{tree-shift-}, namely we have
\[ 
\Sl ^n e_u = \sum_{v\in\Chi^n(u)} \prod_{j=0}^{n-1}\lambda_{\Par^j(v)}\cdot e_v \quad (u\in V, n\in\N)
\]
and
\[ 
\Sl ^{*n} e_u = \left\{\begin{matrix}
\prod_{j=0}^{n-1}\lambda_{\Par^j(u)}\cdot e_{\Par^n(v)}, & \textit{if } \Par^n(u) \textit{ makes sense},\\
0, & \textit{otherwise} 
\end{matrix}\right.  \; (u\in V, n\in\N).
\]
Now the asymptotic limit $A$ of $\Sl$ can be easily obtained.

\begin{lem} \label{aslim_lem}
Let $\Sl$ be a contractive weighted shift on $\irT$. Then the limits 
\[ \alpha_u := \lim_{n\to\infty} \sum_{v\in \Chi^n(u)}\prod_{j=0}^{n-1}\lambda_{\Par^j(v)}^2 \in [0,1]  \quad (u\in V) \] 
exist, and we have
\[A e_u = \alpha_u e_u \quad (u\in V).\]
\end{lem}

\begin{proof}
For every $n\in\N$ and $u\in V$ the following holds:
\[ \Sl ^{*n}\Sl ^n e_u = \sum_{v\in\Chi^n(u)} \prod_{j=0}^{n-1}\lambda_{\Par^j(v)}\cdot \Sl ^{*n}e_v = \sum_{v\in\Chi^n(u)} \prod_{j=0}^{n-1}\lambda_{\Par^j(v)}^2 \cdot e_{u}. \]
Since $\{\Sl ^{*n}\Sl ^n\}_{n=1}^\infty$ converges in SOT, the number $\alpha_u$ exists for every $u\in V$. 
Furthermore, every $\alpha_u$ lies in [0,1], because $\Sl$ is a contraction. By definition, we get $A e_u = \alpha_u e_u \; (u\in V)$.
\end{proof}

Next we obtain some properties of the structure of the stable subspace of $\Sl$ which will be denoted by $\irH_0$ (instead of $\irH_0(\Sl)$) throughout this paper. 
Since $A$ is a diagonal operator, there exists a set $V'\subset V$ such that we have $\irH_0 = \ell^2(V\setminus V')$ and $\irH_0^\perp = \ell^2(V')$.

\begin{prop} \label{stable_subspc_prop}
The following implications are valid for every contractive weighted shift $\Sl$ on $\irT$ and vertex $u\in V$:
\begin{itemize}
\item[\textup{(i)}] if $e_u\in\irH_0$, then $\ell^2(\Des(u))\subseteq\irH_0$ (i.e. $u \notin V' \Longrightarrow \Des(u) \subseteq V\setminus V'$),
\item[\textup{(ii)}] $e_u\in\irH_0$ if and only if $\ell^2(\Chi(u))\subseteq\irH_0$ (i.e. $u \notin V' \iff \Chi(u) \subseteq V\setminus V'$); in particular, we have $\ell^2(\Lea(\irT)) \subseteq \irH_0$,
\item[\textup{(iii)}] if $e_u\in\irH_0^\perp$, then $e_{\Par^k(u)}\in\irH_0^\perp$ for every $k\in\N_0$ (i.e. $u \in V' \Longrightarrow \Par^k(u) \in V', \; \forall \; k\in \N_0$),
\item[\textup{(iv)}] the subgraph $\irT' = (V',E') = (V',E\cap (V'\times V')))$ is a leafless subtree,
\item[\textup{(v)}] if $\irT$ has no root, neither has $\irT'$, and
\item[\textup{(vi)}] if $\irT$ has a root, then either $\Sl\in C_{0\cdot}(\ell^2(V))$ or $\roo_\irT=\roo_{\irT'}$.
\end{itemize}
\end{prop}

\begin{proof}
The fact that $\irH_0$ is invariant for $\Sl$ and that the weights are positive implies (i). 

The sufficiency in (ii) is a part of (i). On the other hand, suppose that $\ell^2(\Chi(u))\subseteq\irH_0$. Then 
\[ \alpha_u = \lim_{n\to\infty} \sum_{w\in \Chi^n(u)}\prod_{j=0}^{n-1}\lambda_{\Par^j(w)}^2 \]
\[ = \lim_{n\to\infty} \sum_{v\in\Chi(u)} \lambda_{v}^2 \sum_{w\in \Chi^{n-1}(v)}\prod_{j=0}^{n-2}\lambda_{\Par^j(w)}^2 = \sum_{v\in\Chi(u)} \lambda_{v}^2 \alpha_{v} =0 \]
is fulfilled, since $\sum_{v\in\Chi(u)} \lambda_{v}^2 \sum_{w\in \Chi^{n-1}(v)}\prod_{j=0}^{n-2}\lambda_{\Par^j(w)}^2 \leq \sum_{v\in\Chi(u)} \lambda_{v}^2 \leq 1$ hold for all $n\in\N$ and $\sum_{w\in \Chi^{n-1}(v)}\prod_{j=0}^{n-2}\lambda_{\Par^j(w)}^2\searrow\alpha_v$. This proves the necessity in (ii). 

Point (iii) follows from (ii) immediately.

Now we turn to the verification of (iv). 
We have to check three conditions for $\irT'$ to be a subtree. 
Two of them are obvious since they were also true in $\irT$. In order to see the connectedness of $\irT'$, two distinct $u',v'\in V'$ are taken.
Since $V = \cup_{j=0}^\infty \Des_\irT(\Par^j_\irT(u'))$, the equation $\Par^k_\irT(u') = \Par^l_\irT(v')$ holds with some $k,l\in\N_0$. 
Then (iii) gives $\Par^i_\irT(u') = \Par^j_\irT(v') \in V'$ for every $i\leq k$ and $j\leq l$, which provides an undirected path in $\irT'$ connecting $u'$ and $v'$.
Finally by (ii) it is trivial that $\irT'$ is leafless.

The last two points immediately follow from (iii).
\end{proof}

In view of (v)-(vi), we have $\Par_{\irT}(u')=\Par_{\irT'}(u')$ for any $u'\in V'$, so we will simply write $\Par(u')$ in this case as well.
We note that in Proposition \ref{stable_subspc_prop} some points are not true, if we allow zero weights.

At the end of this section we identify the asymptotic limit $A_*$ of the adjoint $\Sl^*$. 
The stable subspace of $\Sl^*$ will be denoted by $\irH_0^*$. 
In the sequel we will use the following. If we have sequence $\{s_j\}_{j=1}^\infty \subset (0,1]$, then we have two possibilities. The first one is when the infinite product $\prod_{j=1}^\infty s_j$ is convergent, in this case we have unconditional convergence, i.e. $\prod_{j=1}^\infty s_{\sigma(j)} = \prod_{j=1}^\infty s_j \in (0,\infty)$ for every permutation $\sigma\colon\N\to\N$. The second one is when $\prod_{j=1}^\infty s_j$ is divergent to zero, in this case we have unconditional divergence to zero, i.e. $\prod_{j=1}^\infty s_{\sigma(j)} = 0$ for every permutation $\sigma\colon\N\to\N$.

\begin{prop} \label{dual_aslim_prop}
If $\Sl$ is a contractive weighted shift on the directed tree $\irT$, then the following two points are satisfied:
\begin{itemize}
\item[\textup{(i)}] If $\irT$ has a root, then $\Sl^*$ is stable.
\item[\textup{(ii)}] If $\irT$ is rootless, then
\[ 
h_u := \sum_{v\in\Gen(u)}\prod_{j=0}^\infty\lambda_{\Par^j(v)}\cdot e_v \in\ell^2(V) \quad (u\in V),
\]
and 
\[
\irH_0^{*\perp} = \vee\{h_u\colon u\in V\}.
\]
If $h_u\neq 0$ for some $u\in V$, then this holds for every $u\in V$.
Moreover, in this case we have $h_u = h_v$ if and only if $u\in\Gen(v)$, and the vectors $h_u$ are eigen-vectors:
\[ A_* h_u = a_u h_u \qquad (u\in V) \]
with the corresponding eigen-values
\[ a_u := \|h_u\|^2 = \sum_{v\in\Gen(u)}\prod_{j=0}^\infty\lambda_{\Par^j(v)}^2. \]
\end{itemize}
\end{prop}

\begin{proof}
The first statement is clear, so we only deal with (ii). Since we have 
\[
\sum_{v\in\Gen_n(u)}\prod_{j=0}^{\infty}\lambda_{\Par^j(v)}^2 \leq \sum_{v\in\Gen_n(u)}\prod_{j=0}^{n-1}\lambda_{\Par^j(v)}^2 = \|\Sl^n  e_{\Par^n(u)}\|^2 \leq 1 \quad (n\in\N),
\]
we obtain $h_u\in \ell^2(V)$ for every $u\in V$.
For every $n\in\N$ we compute the following:
\[ 
\Sl ^n\Sl ^{*n}e_u = \prod_{j=0}^{n-1}\lambda_{\Par^j(u)}\cdot \Sl ^n e_{\Par^n(u)} =  \prod_{j=0}^{n-1}\lambda_{\Par^j(u)}\sum_{v\in\Gen_n(u)} \prod_{j=0}^{n-1}\lambda_{\Par^j(v)}\cdot e_v. 
\]
Since $\lim_{n\to\infty}\Sl ^n\Sl ^{*n}e_u = A_*e_u$, we get
\[ 
\langle A_*e_u,e_v\rangle  = \left\{\begin{matrix}
\prod_{j=0}^{\infty}\lambda_{\Par^j(u)}\prod_{j=0}^{\infty}\lambda_{\Par^j(v)} & \text{if } v\in\Gen(u) \\
0 & \text{otherwise}
\end{matrix}\right., 
\]
which yields 
\begin{equation}\label{eq1}
A_*e_u = \prod_{j=0}^{\infty}\lambda_{\Par^j(u)}\sum_{v\in\Gen(u)} \prod_{j=0}^{\infty}\lambda_{\Par^j(v)}\cdot e_v = \prod_{j=0}^{\infty}\lambda_{\Par^j(u)}\cdot h_u \quad (u\in V).
\end{equation}
Therefore we conclude $\irH_0^{*\perp} = \ran(A_*)^- = \vee\{h_u\colon u\in V\}$.

Now, we calculate
\[ 
A_*h_u = \sum_{v\in\Gen(u)}\prod_{j=0}^\infty\lambda_{\Par^j(v)}\cdot A_*e_v = \bigg(\sum_{v\in\Gen(u)}\prod_{j=0}^\infty\lambda_{\Par^j(v)}^2\bigg) h_u \quad (u\in V).
\]
It is easy to see that if we have $h_u = 0$ for some $u\in V$, then $\prod_{j=0}^\infty\lambda_{\Par^j(v)} = 0$ holds for every $v\in V$. But then \eqref{eq1} gives $\irH_0^* = \ell^2(V)$.

Finally, if we have $h_u\neq 0$ for every $u\in V$, then by the definition of $h_u$ it is clear that $h_u = h_v$ holds if and only if $u\in\Gen(u)$.
\end{proof}


\section{Isometric asymptotes of contractive weighted shifts on directed trees}

In this section we want to describe the isometric asymptote of $\Sl$. 
If $C$ is an arbitrary set, then let
\[|C| := \left\{\begin{matrix}
n & \text{if } C \text{ has exactly } n\in \N_0 \text{ elements},\\
\infty & \text{if } C \text{ has infinitely many elements}.
\end{matrix}\right.\]
We call the vertex $u$ a \textit{branching vertex} if $|\Chi(u)| > 1$.
The set of all branching vertices is denoted by $V_\prec$. The quantity 
\[ \Br(\irT) := \sum_{u\in V_\prec} (|\Chi(u)|-1) \in \N_0\cup\{\infty\}\] 
is the \textit{branching index} of $\irT$. By (ii) of \cite[Proposition 3.5.1]{tree-shift} we have 
\begin{equation} \label{co-rank_eq} 
\dim(\ran(\Sl)^\perp) 
= \left\{ \begin{matrix}
1+\Br(\irT) & \text{if } \irT \text{ has a root,} \\
\Br(\irT) & \text{if } \irT \text{ has no root.} 
\end{matrix} \right.
\end{equation}
In Proposition \ref{stable_subspc_prop} we used the notation $\irT' = (V',E')$ for the subtree such that $\ell^2(V') = \irH_0^\perp$. 
We will write $S\in\irB(\ell^2(\Z))$ and $S^+\in\irB(\ell^2(\N_0))$ for the simple bilateral and unilateral shift operators (of multiplicity one), i.e.: $Se_n = e_{n+1}$ $(n\in\Z)$ and $S^+e_k = e_{k+1}$ $(k\in\N_0)$. 
The contraction $T$ is called completely non-unitary (or c.n.u. for short) if the only reducing subspace $\irM$ such that $T|\irM$ is a unitary operator is the trivial $\{0\}$ subspace. By the Sz.-Nagy--Foias--Langer decomposition theorem the contraction $T$ is c.n.u. if and only if $\ker(A-I)\cap\ker(A_*-I) = \{0\}$. From the von Neumann--Wold decomposition it is clear that the c.n.u. isometries are exactly those which are unitarily equivalent to a simple unilateral shift operator (not necessarily of multiplicity one).

\begin{thm} \label{isom_as_thm}
Let us consider a contractive weighted shift $\Sl$ on the directed tree $\irT$ such that $\Sl \notin C_{0\cdot}(\ell^2(V))$.
Then the isometric asymptote $U = S_{\boldsymbol{\beta}}\in\irB(\ell^2(V'))$ is a weighted shift on the subtree $\irT' = (V',E')$ with weights 
\begin{equation}\label{Sb_eq}
\boldsymbol{\beta} = \left\{\beta_{v'} = \frac{\lambda_{v'}\sqrt{\alpha_{v'}}}{\sqrt{\alpha_{\Par(v')}}}\colon v'\in (V')^\circ\right\},
\end{equation}
where $\alpha_{v'}$ is as in Lemma \ref{aslim_lem}.
Moreover, this isometry is unitarily equivalent to the following orthogonal sum:
\begin{itemize}
\item[\textup{(i)}] $\sum_{j=1}^{\Br(\irT')+1}\oplus S^+$, if $\irT$ has a root,
\item[\textup{(ii)}] $\sum_{j=1}^{\Br(\irT')}\oplus S^+$, if $\irT$ has no root and $U$ is a c.n.u. isometry, i.e. when \\ $\sum_{v'\in\Gen_{\irT'}(u')}\prod_{j=0}^\infty\beta_{\Par^j(v')}^2 = 0$ for some $u'\in V'$,
\item[\textup{(iii)}] $S\oplus\sum_{j=1}^{\Br(\irT')}\oplus S^+$, if $\irT$ has no root and $U$ is not a c.n.u. isometry.
\end{itemize}
\end{thm}

\begin{proof}
For any $u'\in V'$ we have the following equation:
\[ 
U e_{u'} = \frac{1}{\sqrt{\alpha_{u'}}}\cdot UA^{1/2}e_{u'} = \frac{1}{\sqrt{\alpha_{u'}}}\cdot A^{1/2}\Sl e_{u'} 
\]
\[ 
= \frac{1}{\sqrt{\alpha_{u'}}}\cdot \sum_{v\in\Chi_\irT(u')}\lambda_{v} \cdot A^{1/2}e_{v} = \sum_{v'\in\Chi_{\irT'}(u')} \frac{\lambda_{'v} \sqrt{\alpha_{v'}}}{\sqrt{\alpha_{u'}}}\cdot e_{v'}. 
\]
This shows that $U = S_{\boldsymbol{\beta}}$ is indeed a weighted shift on $\irT'$ with \eqref{Sb_eq}.

First, we suppose that $\irT$ has a root. Then by Proposition \ref{stable_subspc_prop} $\irT'$ has the same root as $\irT$. But contractive weighted shifts on a directed tree which has a root are of class $C_{\cdot 0}$, so in this case $U$ is unitarily equivalent to a simple unilateral shift operator. Since the co-rank of $U$ is $\Br(\irT')+1$, we infer that $U$ and $\sum_{j=1}^{\Br(\irT')+1}\oplus S^+$ are unitarily equivalent.

Second, we assume that $\irT$ has no root and $U$ is a c.n.u. isometry. 
The isometry $U$ is c.n.u. if and only if $U\in C_{\cdot 0}(\ell^2(V'))$, and by Proposition \ref{dual_aslim_prop} this happens if and only if $\sum_{v'\in\Gen_{\irT'}(u')}\prod_{j=0}^\infty\beta_{\Par^j(v')}^2 = 0$ for some (and then for every) $u'\in V'$. 
Since the co-rank of $U$ is $\Br(\irT')$, the isometry $U$ is unitarily equivalent to $\sum_{j=1}^{\Br(\irT')}\oplus S^+$.

Finally, let us suppose that $\irT$ has no root and $\sum_{v'\in\Gen_{\irT'}(u')}\prod_{j=0}^\infty\beta_{\Par^j(v')}^2 > 0$ for every $u'\in V'$. 
By Proposition \ref{dual_aslim_prop} the unitary part of $U$ clearly acts on the subspace
\[(\irH_0^*(U))^\perp = \bigvee\bigg\{k_{u'} = \sum_{v'\in\Gen_{\irT'}(u')}\prod_{j=0}^\infty\beta_{\Par^j(v')}\cdot e_{v'}\colon u'\in V'\bigg\}.\]
Set $u'\in V'$, then we compute the following: 
\[ U k_{u'} = \sum_{v'\in\Gen_{\irT'}(u')}\prod_{j=0}^\infty\beta_{\Par^j(v')}\cdot Ue_{v'} \]
\[ = \sum_{v'\in\Gen_{\irT'}(u')}\sum_{w'\in\Chi_{\irT'}(v')} \prod_{j=0}^\infty\beta_{\Par^j(v')}\cdot\beta_{w'}e_{w'} \]
\[=  \sum_{w'\in\Gen_{\irT'}(\tilde{w}')} \prod_{j=0}^\infty\beta_{\Par^j(w')}\cdot e_{w'} = k_{\tilde{w}'}\]
with some $\tilde{w}'\in\Chi_{\irT'}(u')$. Therefore we get that $U|\irH_0^*(U)$ is a simple bilateral shift operator. Since the co-rank of $U$ is precisely $\Br(\irT')$, we obtain that $U$ is unitarily equivalent to $S\oplus\sum_{j=1}^{\Br(\irT')}\oplus S^+$.
\end{proof}

\begin{rem} \label{unitarily_eq_rem}
(i) From the theorem above we can calculate the unitary asymptote of $\Sl$. 
In fact, it is the minimal unitary dilation $W$ of the isometry $U$. 
It is easy to see that this minimal unitary dilation is unitarily equivalent to a simple bilateral shift operator of multiplicity $\Br(\irT')$ or $\Br(\irT')+1$.

(ii/a) If the directed tree $\irT$ has a root, then any isometric weighted shift on $\irT$ is of class $C_{\cdot 0}$, i.e.: it is unitarily equivalent to a simple unilateral shift operator with multiplicity $\Br(\irT)$.

(ii/b) In general if we have an isometric weighted shift $U$ on a directed tree, then the structure of the tree does not tell us whether $U$ is a c.n.u. isometry or not. 
To see this take a rootless binary tree (i.e. $|\Chi(u)|=2$ holds for every $u\in V$). 
If we set the weights $\gamma_v := \frac{1}{\sqrt{2}}$ $(v\in V^\circ)$, then $S_{\boldsymbol{\gamma}}$ is clearly an isometry with $\sum_{v\in\Gen_{\irT}(u)}\prod_{j=0}^\infty\gamma_{\Par^j(v)}^2 = \sum_{v\in\Gen_{\irT}(u)}\prod_{j=0}^\infty \frac{1}{2} = 0$ for all $u\in V$. 
Therefore by Theorem \ref{isom_as_thm}, $U$ has to be unitarily equivalent to a simple unilateral shift operator.

On the other hand, let us fix a two-sided sequence of vertices: $\{u_l\}_{l=-\infty}^\infty$ such that $\Par(u_l)=u_{l-1}$ is valid for every $l\in\Z$, and set the following weights:
\[ 
\gamma_v := \left\{\begin{matrix}
\frac{1}{\sqrt{2}} & \text{if } v\in V\setminus \left(\cup_{l=-\infty}^\infty \Chi(u_l)\right), \\
\exp{\frac{-1}{(|l|+1)^2}} & \text{if } v = u_l \text{ for some } l\in\Z, \\
1-\gamma_{u_l}^2 & \text{if } v \in \Chi(u_{l-1})\setminus\{u_l\} \text{ for some } l\in\Z,
\end{matrix} \right..
\]
This clearly defines an isometry $S_{\boldsymbol{\gamma}}$. Since
\[ \sum_{v\in\Gen_{\irT}(u_l)}\prod_{j=0}^\infty\gamma_{\Par^j(v)}^2 \geq \prod_{j=0}^\infty\gamma_{u_{l-j}}^2 = \exp{\left(2\sum_{j=0}^\infty \frac{-1}{(|l-j|+1)^2}\right)} > 0 \quad (l\in\Z), \]
the weighted shift isometry $S_{\boldsymbol{\gamma}}$ is not c.n.u.
\end{rem}

The above points show that two unitarily equivalent weighted shifts on directed trees can be defined on a very different directed tree. 
We close this section by calculating the isometric asymptote of the adjoint $\Sl^*$. 
Namely, we compute the unique isometry $U_*\in\irB((\irH_0^*)^\perp)$ which satisfies the equation $A_*^{1/2}\Sl^* = U_*A_*^{1/2}$.

\begin{thm} \label{dual_isom_as_thm}
Suppose that the contractive weighted shift $\Sl$ on $\irT$ is not of class $C_{\cdot 0}$. 
Then $\irT$ has no root and the isometry $U_*$ acts as follows:
\[ U_* h_{u} = \frac{\sqrt{a_u}}{\sqrt{a_{\Par(u)}}} \cdot h_{\Par(u)} \quad (u\in V), \]
where $0\neq h_u\in\ell^2(V)$ and $a_u\in(0,1]$ are as in Proposition \ref{dual_aslim_prop}. 
As a matter of fact, $U_*$ is unitarily equivalent to a simple unilateral shift operator if there is a last level (i.e. $\Chi(\Gen(u))=\emptyset$ for some $u\in V$), and to a simple bilateral shift operator otherwise.
\end{thm}

\begin{proof}
If $\Sl\notin C_{\cdot 0}(\ell^2(V))$, then $h_u\neq 0$ and $a_u\neq 0$ $(u\in V)$. 
For any $u\in V$ we have the following equation:
\[ U_* \frac{1}{\sqrt{a_u}}h_{u} = \frac{1}{a_u} U_*A_*^{1/2}h_u = \frac{1}{a_u} A_*^{1/2}\Sl^* h_u \]
\[ = \frac{1}{a_u} \sum_{v\in\Gen(u)}\prod_{j=0}^\infty\lambda_{\Par^j(v)} A_*^{1/2}\Sl^* e_v = \frac{1}{a_u} \sum_{v\in\Gen(u)}\prod_{j=0}^\infty\lambda_{\Par^j(v)} \lambda_v A_*^{1/2}e_{\Par(v)} \]
\[ = \frac{1}{a_u}\sum_{v\in\Gen(u)}\prod_{j=0}^\infty\lambda_{\Par^j(v)}\lambda_v \frac{\langle e_{\Par(v)},h_{\Par(v)}\rangle }{\|h_{\Par(v)}\|^2} A_*^{1/2} h_{\Par(v)} \]
\[ = \frac{1}{a_u \sqrt{a_{\Par(u)}}}\bigg(\sum_{v\in\Gen(u)}\prod_{j=0}^\infty\lambda_{\Par^j(v)}^2 \bigg) h_{\Par(u)} = \frac{1}{\sqrt{a_{\Par(u)}}} h_{\Par(u)}.\]
One can easily see the unitary equivalence to the simple uni- or bilateral shift operator.
\end{proof}


\section{Cyclicity of weighted backward shift operators}

The aim of this section is to prove that a weighted backward shift operator of countable multiplicity is cyclic exactly when it has at most one zero weight. In the article \cite{chinese}, written in Chinese, there is a proof for the case when the multiplicity is one, but the author of the present paper was unable to read it due to the lack of proper translation. The reader can consider the forthcoming theorem as a generalization of that result. We note that the forthcoming proof was motivated by the solution of \cite[Problem 160]{Halmos}.

\begin{thm}\label{cyc_backward_thm}
Suppose that $\{e_{j,k}\colon j\in \irJ, k\in\N_0\}$ is an orthonormal basis in $\irH$ where $\irJ \neq \emptyset$ is a countable set and $\{w_{j,k}\colon j\in \irJ, k\in\N_0\}\subseteq [0,\infty)$ is a bounded set of weights. Consider the following weighted backward shift operator (of multiplicity $\#\irJ$):
\[ Be_{j,k} = \left\{ \begin{matrix}
0 & \text{ if } k=0 \\
w_{j,k-1}e_{j,k-1} & \text{ otherwise}
\end{matrix} \right.. \]
Then $B\in\irB(\irH)$ is cyclic if and only if there is at most one zero weight.
\end{thm}

\begin{proof}
Throughout the proof we may always assume without loss of generality that $0\leq w_{j,k}\leq 1$ holds for every $k\in\N_0$ and $j\in\irJ$. 

First we assume that $B$ has only positive weights. We take a vector of the following form:
\[ 
f = \sum_{l = 1}^\infty \xi_{j_l,k_l}\cdot e_{j_l,k_l} \in \irH, 
\]
such that $\xi_{j_l,k_l}>0$ $(l\in\N)$, $0 < k_{l+1}-k_l\nearrow\infty$ and for any $j\in\irJ$ infinitely many $l\in\N$ exist which satisfy $j_l=j$.

Our aim is to modify $f$ by decreasing its non-zero coordinates in a way that they remain positive and after the procedure we obtain a modification of $f$: $\tilde{f}\neq 0$ which is a cyclic vector of $B$. 
We have
\begin{equation} \label{for_cyclicity_eq} 
\begin{gathered}
\displaystyle \frac{1}{\xi_{j_m,k_m} w_{j_m,k_m-1}\dots w_{j_m,k_m-k}}B^k f = e_{j_m,k_m-k} \\
\displaystyle + \sum_{l > m} \frac{\xi_{j_l,k_l}w_{j_l,k_l-1}\dots w_{j_l,k_l-k}}{\xi_{j_m,k_m} w_{j_m,k_m-1}\dots w_{j_m,k_m-k}}\cdot e_{j_l,k_l-k} \quad (m\in\N, k_{m-1} < k \leq k_m)
\end{gathered}
\end{equation}
where we set $k_0 = -1$. 
We consider the quantity 
\begin{equation} \label{cyclic_eq}
\begin{gathered}
\Sigma_m :=
\max_{k_{m-1} < k \leq k_m} \left\{ \sum_{l > m} \bigg|\frac{1}{\xi_{j_m,k_m} w_{j_m,k_m-1}\dots w_{j_m,k_m-k}}B^k f - e_{j_m,k_m-k}\bigg|^2 \right\}\\
= \max_{k_{m-1} < k \leq k_m} \left\{ \sum_{l > m} \bigg|\frac{\xi_{j_l,k_l}w_{j_l,k_l-1}\dots w_{j_l,k_l-k}}{\xi_{j_m,k_m} w_{j_m,k_m-1}\dots w_{j_m,k_m-k}}\bigg|^2\right\}.
\end{gathered}
\end{equation}
Let us suppose for a moment that $\Sigma_m\leq (1/2)^m$ is satisfied for all $m\in\N$. 
In this case $e_{j,k} \in \irH_{B,f}$ would hold for every $j\in\irJ, k\in\N_0$, and thus $f$ would be a cyclic vector for $B$. 
Therefore our aim during the modification process is that this inequality will hold for the modified vector for every $m\in\N$.

If $\Sigma_1 > 1/2$, then let us change every $\xi_{j_l,k_l}$ to $\frac{\xi_{j_l,k_l}}{\sqrt{2\Sigma_1}}$ for every $l > 1$, otherwise we do not do anything. 
Then with these modified coordinates $\Sigma_1 \leq 1/2$ is fulfilled. 
If $\Sigma_2 > 1/4$, then we change every $\xi_{j_l,k_l}$ to $\frac{\xi_{j_l,k_l}}{\sqrt{4\Sigma_2}}$ for every $l > 2$, otherwise we do not modify anything. 
Then $\Sigma_1$ becomes less or equal than before and $\Sigma_2 \leq 1/4$ is satisfied \dots Suppose that we have already achieved $\Sigma_j \leq 1/2^j$ for every $1 \leq j \leq m-1$. 
In case when $\Sigma_m > 1/2^m$, we modify $\xi_{j_l,k_l}$ to $\frac{\xi_{j_l,k_l}}{\sqrt{2^m\Sigma_m}}$ for every $l > m$. 
Otherwise we do not change anything. 
Then $\Sigma_j$ becomes less or equal than before for every $1 \leq j \leq m-1$ and $\Sigma_m \leq 1/2^m$ \dots and so on. 
We notice that every coordinate was modified only finitely many times. 
Therefore this procedure gives us a new vector $\tilde f$ which satisfies $\Sigma_m\leq(1/2)^m$ ($m\in\N$). 
Therefore $\tilde f$ is cyclic for the injective weighted backward shift operator $B$.

Next, we proceed with the case when there is exactly one zero weight. 
Then $B$ is unitarily equivalent to $B'\oplus N$ where $B'\in\irB(\irH)$ is a weighted backward shift operator with positive weights and $N\in\irB(\C^n)$ is a cyclic nilpotent operator.
Since $B'$ has dense range and it is cyclic, (iii) of Lemma \ref{cyclic_denserenage_nilp_lem} gives us what we wanted.

Finally, the necessity is clear, since the co-dimension of $\ran(B)^-$ is at most one whenever $B$ is cyclic.
\end{proof}

We note that using the above method we can also prove the following: if $B$ is injective and there is a vector $g\in\cap_{n = 1}^\infty \ran(B^n)$ such that for every fixed index $j\in\irJ$ the condition $\langle g,e_{j,k}\rangle  \neq 0$ is fulfilled for infinitely many $k\in\N_0$, then there is a cyclic vector $f$ from the linear manifold $\cap_{n = 1}^\infty \ran(B^n)$.

We close this section with the following easy consequence.

\begin{cor}
If the operator $B$ defined in the previous theorem is a $C_{\cdot 1}$ contraction, then $B$ is a cyclic operator.
\end{cor}

\begin{proof}
Clearly, $B^*$ is injective which implies that every weight is non-zero.
\end{proof}


\section{Cyclicity of $\Sl$}

In this section we will deal with cyclic properties of the operator $\Sl$. 
From equation (\ref{co-rank_eq}) we infer that if $\irT$ has a root and $\Br(\irT)>0$, or $\irT$ is rootless and $\Br(\irT)>1$, then the weighted shift on $\irT$ has no cyclic vectors, since in this case the co-rank of $\Sl$ is greater than 1.

When $\Br(\irT)=0$, the operator $\Sl$ is either a cyclic nilpotent operator acting on a finite dimensional space or a weighted bilateral, unilateral or backward shift operator. 
We have dealt with the backward case. 
It is easy to see that a weighted unilateral shift operator is cyclic if and only if it is injective. 
For the weighted bilateral shift operator both can happen, as it was mentioned at the end of Section 2.

So the only interesting pure weighted shift on a directed tree case is when $\Br(\irT)=1$ and $\irT$ has no root.
Of course there are three different cases.
When $\irT$ has exactly two leaves, then obviously it can be represented by the following graph: $\irT^2_{k_0,j_0} := (V^2_{k_0,j_0},E^2_{k_0,j_0})$ where $1\leq k_0\leq j_0$, $V^2_{k_0,j_0} := \{\dots,-2,-1,0,1,\dots j_0\}\cup\{1',\dots k'_0\}$, and $E^2_{k_0,j_0} := \{(k,k+1)\colon k\in\Z, k< j_0\}\cup\{(0,1')\}\cup\{(k',(k+1)')\colon k\in\N, k<k_0\}$ (see Figure 1).
When $\irT$ has exactly one leaf, then it will be represented by the graph $\irT^1_{k_0} := (V^1_{k_0},E^1_{k_0})$ where $k_0\in\N$, $V^1_{k_0} := \Z\cup\{1',\dots k'_0\}$, and $E^1_{k_0} := \{(k,k+1)\colon k\in\Z\}\cup\{(0,1')\}\cup\{(k',(k+1)')\colon k\in\N, k<k_0\}$ (see Figure 1).
Finally, when $\irT$ has no leaf, then it will be represented by the following graph: $\irT^0 := (V^0,E^0)$ where $V^0 := \Z\cup\{k'\colon k\in\N\}$, and $E^0 := \{(k,k+1)\colon k\in\Z\}\cup\{(0,1')\}\cup\{(k',(k+1)')\colon k\in\N\}$. In the next two lemmas our aim is to find a weighted bilateral/backward shift operator and a cyclic nilpotent operator such that their orthogonal sum is similar to $\Sl$. In order to do this, we will construct a bounded invertible operator which intertwines them.

\begin{figure} \label{sim_lem_pic}
\centering
\includegraphics[scale=0.4]{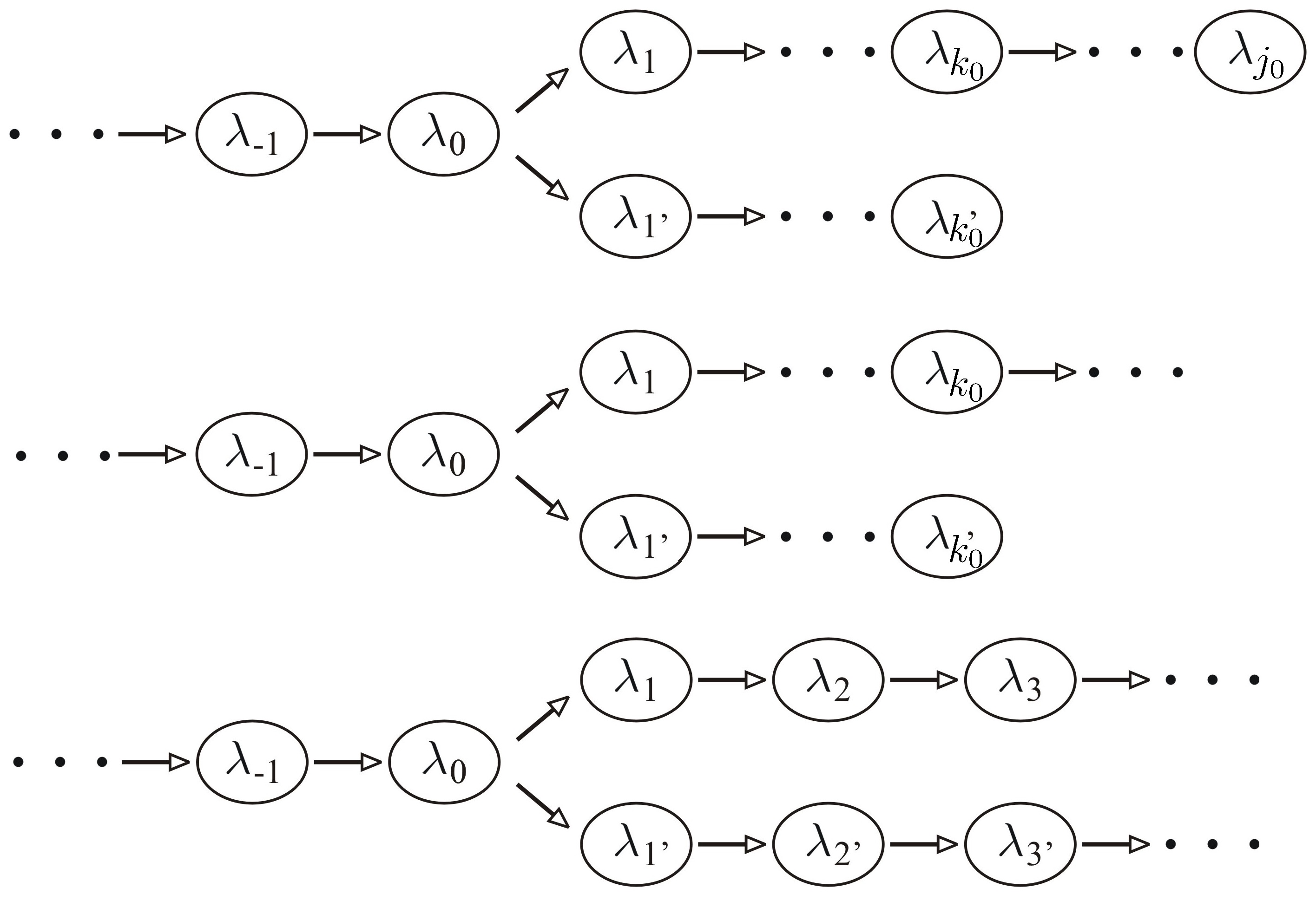}
\caption{Weighted shifts defined on $\irT^2_{k_0,j_0}$, $\irT^1_{k_0}$ and $\irT^0$.}
\end{figure}

\begin{lem} \label{br2_leaf1_similarity_lem}
Consider a bounded weighted shift operator $\Sl$ on $\irT^1_{k_0}$. 
Then $\Sl$ is similar to an orthogonal sum $W\oplus N$ where $W\in\irB(\ell^2(\Z))$ is the weighted bilateral shift operator: $We_k = \lambda_{k+1}e_{k+1}$ $(k\in\Z)$, and $N$ is a cyclic nilpotent operator acting on a finite dimensional space.
\end{lem}

\begin{proof}
Let us define the subspaces $\irE := \ell^2(\Z)$, $\irE' := \irE^\perp = \ell^2(\{1',2',\dots k_0'\})$, and the vectors $g_k := \prod_{j=1}^k \frac{1}{\lambda_j}\cdot e_k - \prod_{j=1}^k \frac{1}{\lambda_{j'}}\cdot e_{k'}$ $(1\leq k\leq k_0)$. 
We consider the weighted bilateral shift operator $W\in\irB(\irE),\; We_k = \lambda_{k+1}e_{k+1}$ $(k\in\Z)$, and the cycllic nilpotent operator
\[ 
N\in \irB(\irE'), \quad e_{k'}\mapsto 
\left\{ \begin{matrix}
\frac{\|g_{k}\|}{\|g_{k+1}\|}e_{(k+1)'} & \text{ if } 1\leq k < k_0\\
0 & \text{ if } k = k_0
\end{matrix}\right.. 
\]
The operator
\[ X \colon \ell^2(V^1_{k_0}) \to \ell^2(V^1_{k_0}), \quad e_{k'} \mapsto \frac{1}{\|g_k\|}g_k \; (1\leq k\leq k_0), \; e_n\mapsto e_n \; (n\in\Z) \]
is trivially bounded. 
Since $\vee\{e_k, e_{k'}\}$ is invariant for $X$ $(1 \leq k \leq k_0)$, and $e_n$ is an eigenvector $(n \in\Z)$, the invertibility of $X$ is also obvious. 
The following observations show that $X(W\oplus N)^* = \Sl^* X$ is satisfied as well:
\[ 
X(W\oplus N)^* e_n = X \lambda_{n} e_{n-1} = \lambda_{n} e_{n-1} = \Sl^* e_n = \Sl^* X e_n \quad (n \in \Z),
\]
\[ 
X(W\oplus N)^* e_{1'} = 0 = \Sl^* \Big(\frac{1}{\|g_1\|}g_1\Big) = \Sl^* X e_{1'},
\]
\[ 
X(W\oplus N)^* e_{k'} = X \frac{\|g_{k-1}\|}{\|g_{k}\|} e_{(k-1)'} = \frac{1}{\|g_{k}\|} g_{k-1} 
\]
\[ 
= \Sl^* \Big(\frac{1}{\|g_k\|}g_k\Big) = \Sl^* X e_{k'} \quad (2\leq k\leq k_0). 
\]
Therefore the invertible operator $X^*$ intertwines $\Sl$ with $W\oplus N$, which ends the proof.
\end{proof}

The verification of the next lemma is quite the same as the proof of the above lemma. Therefore we omit its proof.

\begin{lem} \label{br2_leaf2_similarity_lem}
Consider a bounded weighted shift $\Sl$ on $\irT^2_{k_0,j_0}$. 
Then $\Sl$ is similar to an orthogonal sum $B\oplus N$ where $B\in\irB(\ell^2(\Z\cap (-\infty,j_0])$ is the weighted backward shift operator: 
\[
Be_k = \left\{
\begin{matrix}
\lambda_{k+1}e_{k+1} & \text{if } k < j_0\\
0 & \text{if } k = j_0
\end{matrix}
\right.,
\] 
and $N$ is a cyclic nilpotent operator acting on a finite dimensional space.
\end{lem}

The next theorem is a trivial consequence of Lemma \ref{cyclic_denserenage_nilp_lem} and \ref{br2_leaf2_similarity_lem}, and Theorem \ref{cyc_backward_thm}, so we omit its proof.

\begin{thm}
Every bounded weighted shift operator $\Sl$ on $\irT^2_{k_0,j_0}$ is cyclic.
\end{thm}

Now, let $\T$ denote the complex unit circle, $\irL$ the $\sigma$-algebra of Lebesgue measurable sets on $\T$, $m$ the normalized Lebesgue measure on $\T$, and $L^2$ the space $L^2 = L^2 (\T,\irL,m)$. The simple bilateral shift operator (of multiplicity one) $S$ can be represented as a multiplication operator by the identity function $\chi(\zeta) = \zeta$ on $L^2$. It is a known fact that $g\in L^2$ is cyclic for $S\in\irB(L^2)$ if and only if $g(\zeta) \neq 0$ a.e. $\zeta\in\T$ and $\int_\T \log|g| \,dm =-\infty$.

In the next Theorem, we characterize cyclicity of $\Sl$ on $\irT^1_{k_0}$.

\begin{thm}\label{1leaf_cyc_char_thm}
A bounded weighted shift $\Sl$ on $\irT^1_{k_0}$ is cyclic if and only if the weighted bilateral shift operator $W\in\irB(\ell^2(\Z))$, $We_k = \lambda_{k+1} e_{k+1} \; (k\in\Z)$ is cyclic. In particular, if $\Sl$ is contractive and $\Sl\notin C_{\cdot 0}(\ell^2(V))$, then $\Sl$ is cyclic.
\end{thm}

\begin{proof}
By Lemma \ref{br2_leaf1_similarity_lem}, $\Sl$ is similar to $W\oplus N$. If $W$ has no cyclic vectors, then obviously neither has $\Sl$. If $W$ is cyclic, then by Lemma \ref{cyclic_denserenage_nilp_lem} we can obviously see that $\Sl$ has a cyclic vector. Since $C_{\cdot 1}$-class weighted bilateral shift operators are cyclic, the other statement follows immediately.
\end{proof}

The simple unilateral shift operator $S^+$ can also be represented as a multiplication operator by $\chi$, but on the Hardy space $H^2 = \vee\{\chi^n\colon n\in\N_0\} \subseteq L^2$. 
A function $g\in H^2$ is cyclic for $S^+\in\irB(H^2)$ if and only if $g$ is an outer function. 
We proceed with verifying that the orthogonal sum $S\oplus S^+$ has no cyclic vectors. 
This needs only elementary Hardy space techniques.

\begin{prop}\label{SoplusS+_prop}
The operator $S\oplus S^+ \in \irB(L^2\oplus H^2)$ has no cyclic vectors.
\end{prop}

\begin{proof}
Suppose that $f\oplus g\in L^2\oplus H^2$ is a cyclic vector, and let us denote the orthogonal projection onto $L^2\oplus \{0\}$ by $P_1$. Then $\vee\{\chi^nf\colon n\in\N_0\} = P_1(\vee\{\chi^nf\oplus\chi^ng\colon n\in\N_0\})$ is dense in $L^2$ i.e.: $f$ is cyclic for $S$. Similarly we get that $g$ is cyclic for $S^+$. This implies that $f(\zeta)\neq 0$ for a.e. $\zeta\in\T$ and $g$ is an outer function. We show that $0\oplus g\notin (L^2\oplus H^2)_{S\oplus S^+,f\oplus g}$. To see this consider an arbitrary complex polynomial $p$. Then we have
\[ \|(pf)\oplus (pg) - 0\oplus g\|^2 = \int_\T |pf|^2+|(p-1)g|^2 \, dm. \]

One of the sets $A=p^{-1}(\{z\in\T\colon\re{z}<1/2\})$ or $\T\setminus A = p^{-1}(\{z\in\T\colon\re{z}\geq 1/2\})$ has Lebesgue measure at least 1/2. If $m(A)\geq 1/2$ is satisfied, then
\[ \|(pf)\oplus (pg) - 0\oplus g\|^2 \geq \int_A |(p-1)g|^2 \, dm = \int_A |g|^2/4 \, dm \]
\[ \geq \frac{1}{4}\inf\left\{\int_E |g|^2 \, dm\colon E\in\irL, m(E)\geq 1/2 \right\} > 0 \quad (p\in\irP_\C).\]
Similarly if $m(\T\setminus A)\geq 1/2$, then
\[ \|(pf)\oplus (pg) - 0\oplus g\|^2 \geq \frac{1}{4}\inf\left\{\int_E |f|^2 \, dm\colon E\in\irL, m(E)\geq 1/2 \right\} > 0 \; (p\in\irP_\C).\]
These imply that $S\oplus S^+$ has no cyclic vectors.
\end{proof}

Now we are in a position to prove a non-cyclicity theorem.

\begin{thm}\label{T0_nocyc}
If the contractive weighted shift $\Sl$ on $\irT^0$ is of class $C_{1\cdot}$, then it has no cyclic vectors.
\end{thm}

\begin{proof} 
By Theorem \ref{isom_as_thm}, the isometric asymptote $U$ of $\Sl$ is unitarily equivalent to the orthogonal sum $S\oplus S^+$ which has no cyclic vectors by Proposition \ref{SoplusS+_prop}. This implies - together with (i) of Lemma \ref{is_as_forcyclem} - that neither has $\Sl$.
\end{proof}

We note that there exists a weighted shift operator on $\irT^0$ which is cyclic. This will be proven in the last section.


\section{Cyclicity of $\Sl^*$}

In this section we are interested in giving necessary conditions for $\Sl^*$ to be cyclic. 
Unlike in Section 6, here we do not have any restrictions on the structure of the directed tree $\irT$. 
In fact, we will see that there are several directed trees, other than $\irT^0$, $\irT^1_{k_0}$ or $\irT^2_{k_0,j_0}$, on which we can define a weighted shift $\Sl$ such that $\Sl^*$ is cyclic.
Let us denote the operator $\underbrace{S^+\oplus \dots \oplus S^+}_{k \text{ times}}$ by $S^+_k$ ($k\in\N$) and the orthogonal sum of $\aleph_0$ copies of $S^+$ by $S^+_{\aleph_0}$.
Now we prove the cyclicity of $S\oplus (S^+_k)^*$ when $k\in\N$.

\begin{thm}\label{cyclic_ort_sum_thm}
The operator $S\oplus (S^+_k)^*$ is cyclic for every $k\in\N$.
\end{thm}

\begin{proof}
The method is the following: we intertwine $S\oplus S_k^+$ and $S$ with an injective operator $X\in\irB(L^2\oplus H^2,L^2)$: $SX = X(S\oplus S_k^+)$. Then taking the adjoint of both sides in the equation: $(S^*\oplus (S_k^+)^*)X^* = X^*S^*$, $X^*$ has dense range and $S^*$ is cyclic. This implies the cyclicity of $S^*\oplus (S_k^+)^*$ for any $k \in \N$ by (i) of Lemma \ref{cyclic_denserenage_nilp_lem}, which is unitarily equivalent to $S\oplus (S_k^+)^*$. 

For the $k=1$ case the definition of the operator $X$ is the following:
\[ X\colon L^2\oplus H^2 \to L^2, \quad f\oplus g\mapsto f\varphi + g, \]
where $\varphi\in L^\infty$, $\varphi(\zeta)\neq 0$ for a.e. $\zeta\in\T$ and $\int_\T \log|\varphi(\zeta)| \, d\zeta = -\infty$. An easy estimate shows that $X\in\irB(L^2\oplus H^2,L^2)$. Assume that $0 = f\varphi + g$. On the one hand if $f = 0$ ($g=0$, resp.), then $g = 0$ ($f=0$, resp.) follows immediately. On the other hand, taking logarithms of the absolute values and integrating over $\T$ we get 
\[-\infty < \int_\T \log|g| \, dm = \int_\T \log|f| + \log|\varphi| \, dm \leq \int_\T |f| \, dm + \int_\T \log|\varphi| \, dm = -\infty, \]
which is a contradiction. Therefore $X$ is injective. The equation $SX = X(S\oplus S^+)$ is trivial, thus $S\oplus (S^+)^*$ is indeed cyclic.

Now, let us turn to the case when $k>1$. We will work with induction, so let us suppose that we have already proven the cyclicity of $S\oplus (S^+_{k-1})^*$ for some $k > 1$. Consider the following operator
\[ Y\colon L^2\oplus \underbrace{H^2 \oplus\dots\oplus H^2}_{k \text{ times}} \to L^2\oplus \underbrace{H^2 \oplus\dots\oplus H^2}_{k-1 \text{ times}}, \]
\[ f\oplus g_1\oplus \dots \oplus g_{k} \mapsto (f\varphi+g_1)\oplus g_2\oplus \dots \oplus g_{k}, \]
with the same $\varphi\in L^\infty$ as in the definition of $X$. Obviously $Y$ is bounded, linear and injective, and we have $Y(S\oplus S^+_{k}) = (S\oplus S^+_{k-1})Y$. This proves that $S\oplus (S^+_{k})^*$ is also cyclic.
\end{proof}

Of course, now a question arises naturally. It seems that the previous method does not work for the $k=\aleph_0$ case.

\begin{ques}\label{ques}
Is the operator $S\oplus (S^+_{\aleph_0})^*$ cyclic?
\end{ques}

If $\Sl$ is of class $C_{1\cdot}$, then in some cases we can prove that $\Sl^*$ is cyclic.

\begin{thm}\label{adjcyc}
The following conditions are valid:
\begin{itemize}
\item[\textup{(i)}] If $\irT$ has a root and the contractive weighted shift $\Sl$ on $\irT$ is of class $C_{1\cdot}$, then $\Sl^*$ is cyclic.
\item[\textup{(ii)}] If $\irT$ is rootless, $\Br(\irT)<\infty$ and the contractive weighted shift $\Sl$ on $\irT$ is of class $C_{1\cdot}$, then $\Sl^*$ is cyclic.
\end{itemize}
\end{thm}

\begin{proof}
Obviously $\irT$ is leafless in both cases. We consider the isometric asymptote $U$ of $\Sl$ which was described in Theorem \ref{isom_as_thm}. Since $U^*$ is cyclic by Theorems \ref{cyc_backward_thm} and \ref{cyclic_ort_sum_thm}, the operator $\Sl^*$ is also cyclic by Lemma \ref{is_as_forcyclem}.
\end{proof}

If Question \ref{ques} had a positive answer, then in the previous theorem we would only have to assume that $\Sl\in C_{1\cdot}(\ell^2(V))$.
This would be a nice improvement.

To close this section we point out that the condition $\Sl\in C_{1\cdot}(\ell^2(V))$ in the above theorem is crucial. 
In fact, there exists a weighted shift operator on $\irT^0$ such that $\Sl^*$ is not cyclic (see the next section).


\section{Similarity of $\Sl$ to the orthogonal sum of a bi- and a unilateral shift operator}

In the last section we examine the case when $\Sl$ is a weighted shift on $\irT^0$. 
Here we will not assume contractivity of $\Sl$. 
First we will investigate when $\Sl$ is similar to an orthogonal sum of a bi- and a unilateral shift operator. 
Then, as a counterpart of Theorem \ref{T0_nocyc}, we construct a weighted shift operator on $\irT^0$ which is cyclic, and, as a counterpart of Theorem \ref{adjcyc}, we point out that there is another one such that its adjoint $\Sl^*$ has no cyclic vectors. 
We note that in \cite{tree-shift} the directed tree $\irT^0$ was denoted by $\irT_{2,\infty}$.

Let $\boldsymbol{w} = \{w_n\colon n\in\Z\}\cup\{w_{k'}\colon k\in\N\setminus\{1\}\} \subseteq (0,\infty)$ be bounded. 
We define the operator $\Ww\in\irB(\ell^2(V^0))$ by the following equations:
\[ \Ww e_{n} = w_{n+1} e_{n+1}, \qquad \Ww e_{k'} = w_{(k+1)'} e_{(k+1)'} \quad (n\in\Z, k\geq 1). \]
Obviously $\Ww$ is an orthogonal sum of a bi- and a unilateral shift operator.

Our aim is to find out whether there exists an operator $\Ww$ such that it is similar to $\Sl$. In order to do this, we will try to find a bounded invertible operator which intertwines $\Sl$ with a $\Ww$. 
We will use the following notations:
\begin{equation}\label{eq2}
g_k := \frac{1}{\prod_{j=1}^k\lambda_j}\cdot e_k - \frac{1}{\prod_{j=1}^k\lambda_{j'}}\cdot e_{k'}, \;\;
\widetilde{g}_k := \prod_{j=1}^k \lambda_j \cdot e_k + \prod_{j=1}^k \lambda_{j'} \cdot e_{k'} \;\; (k\in\N).
\end{equation}
We also set the following subspaces:
\[ \irE := \vee\{e_k\colon k\in\Z\}, \quad \irE' := \vee\{e_{k'}\colon k\in\N\}, \quad \irG := \vee\{g_k\colon k\in\N\}.\]
Clearly, we have $\irG^\perp = \vee\{\widetilde{g}_k, e_{1-k}\}_{k=1}^\infty$, $\irE^\perp = \irE'$.

\begin{lem}\label{dir_sum_lem}
The following two conditions are equivalent:
\begin{itemize}
\item[\textup{(i)}] the positive sequence $\left\{\prod_{j=1}^k \frac{\lambda_{j'}}{\lambda_j}\colon k\in\N\right\}$ is bounded,
\item[\textup{(ii)}] $\ell^2(V^0) = \irE\dotplus\irG$ (direct sum).
\end{itemize}
\end{lem}

\begin{proof}
It is clear that $\irE\cap\irG = \{0\}$. Therefore (ii) is equivalent to the condition $\irE+\irG = \ell^2(V)$. By \cite[Theorem 2.1]{CoMa} this holds if and only if 
\begin{equation}\label{eqM}
M := \sup\left\{|\langle \widetilde{g}, e' \rangle| \colon \widetilde{g} \in \irG^\perp, e'\in\irE', \|\widetilde{g}\| = \|e'\| = 1\right\} < 1.
\end{equation}
The general form of unit vectors in $\irG^\perp$ and $\irE'$ are the following:
\[
\widetilde{g} = \sum_{k=1}^\infty (\epsilon_{1-k} e_{1-k} + \gamma_k \tfrac{1}{\|\widetilde{g}_k\|}\widetilde{g}_k)\in\irG^\perp, \quad \sum_{k=1}^\infty (|\epsilon_{1-k}|^2 + |\gamma_k|^2) = 1,
\] 
and 
\[
e' = \sum_{k=1}^\infty \epsilon_{k'} e_{k'}\in\irE', \quad \sum_{k=1}^\infty |\epsilon_{k'}|^2 = 1.
\] 
We have
\[
|\langle\widetilde{g},e'\rangle| = 
\left|\sum_{k=1}^\infty \frac{\gamma_k \overline{\epsilon_{k'}}}{\|\widetilde{g}_k\|} \langle\widetilde{g}_k,e_{k'}\rangle\right| =
\]
\[
\left|\sum_{k=1}^\infty \frac{\gamma_k \overline{\epsilon_{k'}}}{\sqrt{\prod_{j=1}^k\lambda_j^2 + \prod_{j=1}^k\lambda_{j'}^2}} \prod_{j=1}^k\lambda_{j'}\right| =
\left|\sum_{k=1}^\infty \frac{\gamma_k \overline{\epsilon_{k'}}}{\sqrt{\left(\prod_{j=1}^k \frac{\lambda_j}{\lambda_{j'}}\right)^2 + 1}}\right|.
\]
Therefore it is straightforward that \eqref{eqM} is equivalent to (i). 
\end{proof}

\begin{figure}
\centering
\includegraphics[scale=0.4]{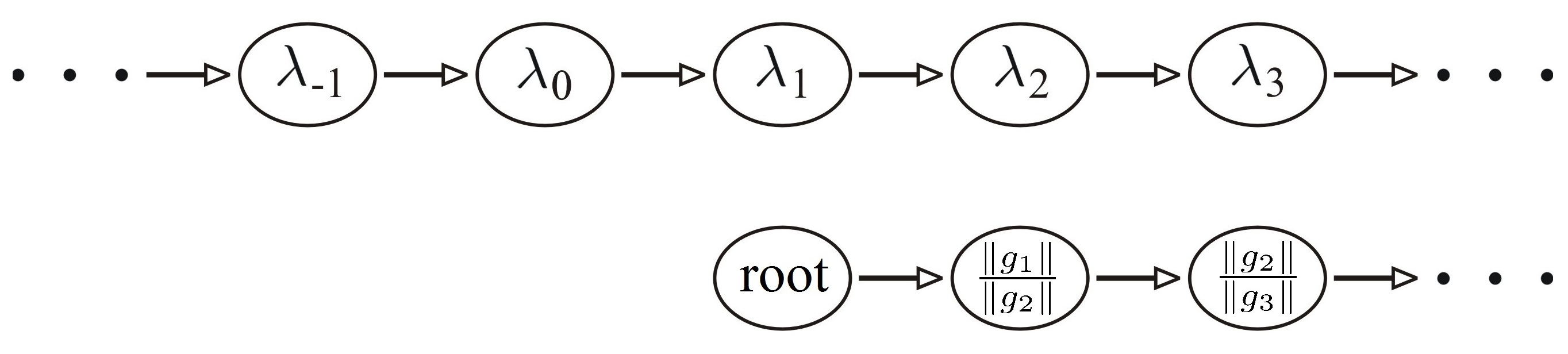}
\caption{In some cases $\Sl$ on $\irT^0$ is similar to an orthogonal sum.}
\end{figure}

Now, we are able to prove a similarity result. The operator $T_1\in\irB(\irH)$ is a \textit{quasiaffine transform} of $T_2\in\irB(\irK)$ if there exists a quasiaffinity (i.e.: which is injective and has dense range) $X\in\irB(\irH,\irK)$ such that $XT_1 = T_2X$. 

\begin{prop}\label{sim_br_1_prop}
Let $\Sl \in \irB(\ell^2(V^0))$ be a weighted shift on the directed tree $\irT^0$ and set
\[ w_n := \lambda_n \; (n\in\Z), \quad  w_{k'} := \frac{\|g_{k-1}\|}{\|g_k\|} \; (k>1), \]
$\boldsymbol{w} := \{w_n\colon n\in\Z\}\cup\{w_{k'}\colon k\in\N\setminus\{1\}\}$, where $g_k$ is as in \eqref{eq2}.
Then $\Ww\in\irB(\ell^2(V^0))$ and the following two points hold:
\begin{itemize}
\item[\textup{(i)}] $\Sl$ is always a quasiaffine transform of $\Ww$.
\item[\textup{(ii)}] If $\left\{\prod_{j=1}^k \frac{\lambda_{j'}}{\lambda_j}\colon k\in\N\right\}$ is bounded, then $\Sl$ is similar to $\Ww$.
\end{itemize}
\end{prop}

\begin{proof}
(i): Since $\Sl$ is bounded and $g_{k-1} = \Sl^*g_k$ $(k>1)$, we have $w_{k'} = \frac{\|g_{k-1}\|}{\|g_k\|} \leq \|\Sl^*\|$ $(k>1)$ and hence $\Ww$ is bounded. We define an operator $X$ by the equations
\[ X e_{k'} = \frac{1}{\|g_k\|} g_k, \quad X e_n = e_n \qquad (k\in\N, n\in\Z). \]
The operator $X$ is bounded and quasiaffine, because for every $k\in\N$ the subspace $\vee\{e_k,e_{k'}\}$ is invariant for $X$ and $\|X|\vee\{e_k,e_{k'}\}\| \leq 2$. The next equations show that $X$ intertwines $\Ww^*$ with $\Sl^*$:
\begin{equation}\label{eq3}
\begin{gathered}
\Sl^* X e_n = \Sl^* e_n = \lambda_n e_{n-1} = \lambda_n X e_{n-1} = X\Ww^* e_n \quad (n\in\Z), \\
\Sl^* X e_{k'} = \frac{1}{\|g_k\|} \Sl^* g_k = \left\{ \begin{matrix}
0 & k=1 \\
\frac{1}{\|g_k\|} g_{k-1} & k>1
\end{matrix} \right. = X\Ww^*e_{k'} \quad (k\in\N).
\end{gathered}
\end{equation}
This proves that $\Sl$ is indeed a quasiaffine transform of $\Ww$.

(ii): Clearly, the restrictions $X|\irE'\in\irB(\irE',\irG)$ and $X|\irE\in\irB(\irE,\irE)$ are bijective isometries. 
Since by Lemma \ref{dir_sum_lem} we have $\irE\dotplus\irG = \ell^2(V^0)$, and by definition $\irE\oplus\irE' = \ell^2(V^0)$, the operator $X$ is invertible.
Therefore by \eqref{eq3} we obtain that $\Sl$ is similar to $\Ww$.
\end{proof}

We have the following consequence.

\begin{cor}
If $\Sl \notin C_{0\cdot}(\ell^2(V^0))$ is a contractive weighted shift on the directed tree $\irT^0$, then it is similar to an orthogonal sum of a weighted bi- and a weighted unilateral shift operator.
\end{cor}

\begin{proof}
By Lemma \ref{aslim_lem} and Proposition \ref{stable_subspc_prop} it is easy to see that the condition $\Sl \notin C_{0\cdot}(\ell^2(V^0))$ holds if and only if we have either $\prod_{j=1}^\infty \lambda_j > 0$ or $\prod_{j=1}^\infty \lambda_{j'} > 0$. By interchanging $\lambda_{j'}$ and $\lambda_j$ for every $j\in\N$, if necessary, we can assume that the first inequality is satisfied. Then the sequence $\left\{\prod_{j=1}^k \frac{\lambda_{j'}}{\lambda_j}\colon k\in\N\right\}$ is obviously bounded. Applying the previous proposition, we get the similarity.
\end{proof}

In our last theorem we show that a weighted shift on $\irT^0$ can be cyclic.

\begin{thm}\label{cyclic_br1_no_leaf}
There is a weighted shift $\Sl\in\irB(\ell^2(V^0))$ on $\irT^0$ which is cyclic.
\end{thm}

\begin{proof}
There exists a hypercyclic weighted bilateral shift operator $R\in\irB(\irE), R e_j = \lambda_{j+1} e_{j+1}, \lambda_j>0\; (j\in\Z)$ (see \cite{hypercyclic_bil}). 
Let us define $\{\lambda_{k'}\}_{k=1}^\infty$ recursively such that $\lambda_{1'} = \lambda_1$, $0<\lambda_{k'}\leq\lambda_k$ and $\|g_{k-1}\|\leq\|g_k\|$ holds for every $k\in\N\setminus\{1\}$, where $g_k$ is as in \eqref{eq2}.
Clearly, this can be done.
An application of Proposition \ref{sim_br_1_prop} gives that the weighted shift $\Sl\in\irB(\ell^2(V^0))$ on $\irT^0$ is similar to the operator $\Ww\in\irB(\ell^2(V^0))$.
By definition, the unilateral summand of $\Ww$ is a contraction, but the bilateral summand $R$ is not (otherwise it would not be hypercyclic).
It is enough to show that $\Ww$ is cyclic, in fact, we will show that $f\oplus e_{1'} \in \irE\oplus \irE'$ is a cyclic vector of $\Ww$ whenever $f\in\irE$ is a hypercyclic vector of $R$.

First, let us take an arbitrary vector $e\in\irE$. An easy observation shows that there is a sequence $\{\Ww^{j_k}(f\oplus 0)\}_{k=1}^\infty$ such that $j_k\in\N$ and $\frac{1}{k}\Ww^{j_k}(f\oplus 0) \to e\oplus 0$. Since the unilateral summand is a contraction, we get $\frac{1}{k}\Ww^{j_k}(f\oplus e_{1'}) \to e\oplus 0$. This implies $\irE\subseteq \vee\{ \Ww^k(f\oplus e_{1'}) \colon k\in\N_0\}$.

Second, we fix a number $n\in\N$. Our aim is to prove that $0\oplus e_{n'} \in \vee\{ \Ww^k(f\oplus e_{1'}) \colon k\in\N_0\}$. 
Since $\Ww^{n-1} (f\oplus 0) \in \irE$, we have $\Ww^{n-1} (f\oplus 0) \in \vee\{ \Ww^k(f\oplus e_{1'}) \colon k\in\N_0\}$.
Therefore we obtain $\Ww^{n-1}(f\oplus e_{1'}) - \Ww^{n-1} (f\oplus 0) = \Ww^{n-1} (0\oplus e_{1'}) \in \vee\{ \Ww^k(f\oplus e_{1'}) \colon k\in\N_0\}$.
Since $\Ww^{n-1} (0\oplus e_{1'})$ is a non-zero scalar multiple of $0\oplus e_{n'}$, our proof is complete.
\end{proof}

Finally, let $R\in\irB(\irE), R e_j = \lambda_{j+1} e_{j+1}, 0<\lambda_j\leq 1\; (j\in\Z)$ such that $R$ has no cyclic vectors (see e.g. \cite{non-cyclic_bil_shift}). 
We set $\lambda_{k'} = \lambda_k$ $(k\in\N)$. 
Then $\Sl\in\irB(\ell^2(V^0))$ has no cylcic vectors, since it is similar to $\Ww$ and the bilateral summand of $\Ww$ is non-cyclic.
This provides a weighted shift $\Sl\in\irB(\ell^2(V^0))$ on $\irT^0$ such that its adjoint $\Sl^*$ has no cyclic vectors.


\section*{Acknowledgement} 
The author expresses his sincere thank to the anonymous referee for his/her extremely thorough review and comments on the first version of the paper which helped the author substantially improve the quality of the presentation.
The author is also grateful to professor L\'aszl\'o K\'erchy for his useful suggestions.

\end{document}